\documentclass[11pt]{article}

\usepackage{amsmath,amssymb,amsthm}
\usepackage[utf8]{inputenc}

\newcommand{\vertiii}[1]{{\vert\kern-0.25ex\vert\kern-0.25ex\vert #1 
    \vert\kern-0.25ex\vert\kern-0.25ex\vert}}
\newcommand{\Vertiii}[1]{{\Big\vert\kern-0.25ex\Big\vert\kern-0.25ex\Big\vert #1 
    \Big\vert\kern-0.25ex\Big\vert\kern-0.25ex\Big\vert}}

\newcommand\IN{\mathbb{N}}
\newcommand\IZ{\mathbb{Z}}

\theoremstyle{plain}
\newtheorem{theorem}{Theorem}
\newtheorem{proposition}[theorem]{Proposition}
\newtheorem{corollary}[theorem]{Corollary}

\theoremstyle{definition}
\newtheorem{definition}[theorem]{Definition}
\newtheorem{remark}[theorem]{Remark}
\newtheorem{example}[theorem]{Example}
\newtheorem{question}[theorem]{Question}

\begin{document}

\setcounter{footnote}{-1}

\title{Frequently hypercyclic bilateral shifts}
\author{Karl-G. Grosse-Erdmann\thanks{
\newline \hspace*{4.3mm}
\newline\hspace*{4.3mm}
This work was supported by the Fonds de la Recherche Scientifique - FNRS under grant no. PDR T. 0164.16 
}}
\date{}
\maketitle
~\\[-16mm]
\begin{quote}
\textsc{Abstract}. It is not known if the inverse of a frequently hypercyclic bilateral weighted shift on $c_0(\IZ)$ is again frequently hypercyclic. We show that the corresponding problem for upper frequent hypercyclicity has a positive answer. We characterise, more generally, when bilateral weighted shifts on Banach sequence spaces are (upper) frequently hypercyclic. 
\end{quote}

\section{Introduction}

This paper is motivated by one of the major open problems in Linear Dynamics, see \cite[Question 4.3]{BaGr06} and also \cite[Problem 44]{GMZ16}.

\begin{question}[Bayart, Grivaux 2006] \label{q1}
Let $T$ be an invertible frequently hypercyclic operator. Is $T^{-1}$ frequently hypercyclic?
\end{question}

It is a consequence of the Baire category theorem (via the Birkhoff transitivity theorem) that the corresponding result is true for classical hypercyclicity, see \cite{BaMa09}, \cite{GrPe11}. Since an operator and its inverse share all their periodic points, the result is also true for chaotic operators. However, the Baire category theorem loses its power for frequent hypercyclicity, see \cite[Proposition 4.1]{BaGr06}, hence the interest in Question \ref{q1}. 

There seems to be a consensus that the answer to the question should be negative. And a good (as well as tractable) candidate might be a suitable bilateral weighted shift operator. Bayart and Ruzsa \cite{BaRu15} have recently obtained the somewhat surprising result that a  bilateral weighted shift on $\ell^p(\IZ)$, $1\leq p<\infty$, is frequently hypercyclic if and only if it is chaotic, which then excludes these operators as counter-examples. This leaves the space $c_0(\IZ)$ as a natural underlying space, even more so as shifts on spaces of null sequences have already provided various counter-examples in Linear Dynamics, see \cite{BaGr07}, \cite{BaRu15}, \cite{BMPP16}, \cite{BoGr17}. Invertible frequently hypercyclic bilateral weighted shifts on $c_0(\IZ)$ have been characterised by Bayart and Ruzsa \cite{BaRu15}, but the complexity of the conditions has so far not allowed to decide Question \ref{q1} for such operators; see also \cite[p.~707]{BaRu15}.

Recently, the related notion of upper frequent hypercyclicity has attracted some attention, see \cite{BaRu15}, \cite{BMPP16}, \cite{BoGr17}. Again, the following problem, which appears implicitly in the paper of Bayart and Ruzsa \cite{BaRu15}, is open. 

\begin{question}\label{q2}
Let $T$ be an invertible upper frequently hypercyclic operator. Is $T^{-1}$ upper frequently hypercyclic?
\end{question}

Interestingly, Bayart and Ruzsa show that if $T$ is invertible and frequently hypercyclic then $T^{-1}$ is upper frequently hypercyclic. They also show that the answer to Question \ref{q2} is positive for weighted backward shifts $B_w$ on $\ell^p(\IZ)$, $1\leq p<\infty$, which suggests again shifts on $c_0(\IZ)$ as the next candidates for a counter-example. The main result of this paper is to dispel this hope.

\begin{theorem}\label{t-main}
Let $B_w$ be an invertible weighted backward shift on $c_0(\IZ)$. If $B_w$ is upper frequently hypercyclic then so is $B_w^{-1}$.
\end{theorem}

The proof requires both constructive arguments and an application of the Baire category theorem. Indeed, the fact that Baire is back in force for the notion of upper frequent hypercyclicity was first noticed by Bayart and Ruzsa \cite{BaRu15} and then further developped in \cite{BoGr17}. This, then, might suggest that the answer to Question \ref{q2} is positive for all operators. However, the Birkhoff type theorem for upper frequent hypercyclicity obtained in \cite{BoGr17} is fundamentally non-symmetric, which leaves open the possibility that the answer is nonetheless negative in general. 

The paper is organized as follows. In Section \ref{s-nn} we fix terminology and notation; in particular we associate a space $\widehat{X}$ to a Banach sequence space $X$ in which the canonical unit sequences form a basis. In Section \ref{s-bilwsarb} we characterise (under suitable assumptions) when a weighted backward shift on $\widehat{X}$ is $\mathcal{A}$-hypercyclic for $X$, and this for arbitrary Furstenberg families $\mathcal{A}$. The proof is constructive. In Section \ref{s-fortoin} we show how one may pass from $\mathcal{A}$-hypercyclicity \textit{for} $X$ to $\mathcal{A}$-hypercyclicity \textit{in} $X$ in the case when $\mathcal{A}$ is an upper Furstenberg family; this result might also be of independent interest. Its proof uses the Baire category theorem. As an application we obtain a characterisation of $\mathcal{A}$-hypercyclicity for weighted backward shifts on $X$ for such families, see Section \ref{s-bilwsupp}. This then allows us to prove Theorem \ref{t-main}. In Section \ref{s-bilwsXarb} we obtain a (different) characterisation of $\mathcal{A}$-hypercyclicity for weighted backward shifts on $X$ for arbitrary Furstenberg families. We end with a remark on unilateral weighted backward shifts, see Section \ref{s-uni}.

\section{Terminology and notation}\label{s-nn}

An operator $T$ on a (real or complex) Banach space $X$ is \textit{hypercyclic} if it admits a dense orbit $\{T^nx:n\geq 0\}$, in which case $x$ is called a \textit{hypercyclic vector}; for the operator to be \textit{chaotic} we need, in addition, a dense set of periodic points. The operator $T$ is called \textit{frequently hypercyclic} if there is a vector $x\in X$ such that, for any non-empty open subset $U$ of $X$, the set $\{n\geq 0 : T^nx\in U\}$ has positive lower density, that is, 
\[
\liminf_{N\to\infty} \frac{1}{N+1} \text{card}\{n\leq N : T^nx\in U\} >0;
\]
the vector $x$ is then called a \textit{frequently hypercyclic vector}. If one replaces lower by upper density one arrives at the notion of \textit{upper frequent hypercyclicity}. For an introduction to Linear Dynamics we refer to the monographs \cite{BaMa09} and \cite{GrPe11}.

More generally, let $\mathcal{A}$ be a \textit{Furstenberg family}, that is, a non-empty family of subsets of $\mathbb{N}_0$ such that if $A\in\mathcal{A}$ and $B\supset A$ then $B\in \mathcal{A}$. Then the operator $T$ is called \textit{$\mathcal{A}$-hypercyclic} if there is a vector $x\in X$ (which is then called an \textit{$\mathcal{A}$-hypercyclic vector}) such that, for any non-empty open subset $U$ of $X$, we have that
\[
\{n\geq 0 : T^nx\in U\}\in \mathcal{A}.
\]
Hypercyclicity (respectively frequent hypercyclicity, upper frequent hypercyclicity) is the special case for the Furstenberg family of all infinite sets (respectively of all sets of positive lower density or of all sets of positive upper density). $\mathcal{A}$-hypercyclicity has recently been studied in \cite{BMPP16} and \cite{BoGr17}.

A crucial idea in this paper is to consider operators whose orbits only allow one to approximate vectors from a closed subspace $Z$ of $X$. More precisely, $T$ is called \textit{$\mathcal{A}$-hypercyclic for $Z$} if there is a vector $x\in X$ (called \textit{$\mathcal{A}$-hypercyclic for $Z$}) such that, for any open subset $U$ of $X$ with $U\cap Z\neq\varnothing$, we have that
\[
\{n\geq 0 : T^nx\in U\}\in \mathcal{A}.
\] 
This restricted type of density for orbits was first considered by the author in \cite{Gro87}. The notion should not be confused with the recent concept of subspace hypercyclicity, see \cite{MaMa11}.

A \textit{Banach sequence space over $\IZ$} is a Banach space that is a subspace of the space $\mathbb{K}^{\mathbb{\IZ}}$ of all (real or complex) sequences and such that each coordinate functional $x=(x_n)_{n\in\IZ}\to x_k$, $k\in\IZ$, is continuous. The canonical unit sequences are denoted by $e_n=(\delta_{n,k})_{k\in\IZ}$. 

We will study bilateral weighted backward shifts $B_w$ on sequence spaces over $\IZ$. They are defined by
\[
B_w(x_n)_{n\in\IZ} = (w_{n+1}x_{n+1})_{n\in\IZ},
\]
where $w=(w_n)_{n\in\IZ}$ is a sequence of non-zero scalars. The unweighted shift (with $w_n=1$ for all $n\in\IZ$) is denoted by $B$.

Let $X$ be a Banach sequence space over $\IZ$ in which $(e_n)_{n\in \IZ}$ is a basis. It is well known that
\[
\vertiii{x} = \sup_{m,n\geq 0} \Big\|\sum_{-m\leq k\leq n} x_k e_k\Big\|
\] 
defines an equivalent norm on $X$. Since this value is defined for any sequence of scalars $x=(x_n)$ we may introduce the sequence space
\[
\widehat{X} = \{ x=(x_n)_{n\in\IZ} : \vertiii{x}<\infty\}.
\]
It is easily seen that $\widehat{X}$ is a Banach sequence space that contains $X$ as a closed subspace. The space appears, for example, in Singer \cite[p.~39]{Sin70} or in Bellenot \cite{Bel84}, but does not seem to have been given a name or a universally accepted notation. We have, for example, that $\widehat{c_0(\IZ)} = \ell^\infty(\IZ)$.

If, now, $B_w$ is a weighted backward shift operator on $X$ then a simple calculation shows that $B_w$ is also an operator on $\widehat{X}$. On the other hand, any weighted backward shift operator on $\widehat{X}$ maps $X$ into itself. Moreover, $B_w$ is invertible on $X$ if and only if it is invertible on $\widehat{X}$.

In the final section of this paper we will be looking at sequence spaces over $\IN_0$, with the obvious adaptations of the notions above.

\section{Bilateral shifts on $\widehat{X}$ for arbitrary Furstenberg families}\label{s-bilwsarb}

In the sequel we will write for simplicity a sequence $x=(x_n)$ as a formal series
\[
\sum_{n\in\IZ} x_n e_n
\]
without necessarily requiring the convergence of this series.

\begin{theorem}\label{t-bahcgen}
Let $X$ be a Banach sequence space over $\IZ$ in which $(e_n)_{n\in \IZ}$ is a basis. Suppose that the backward shift $B$ is an operator on $X$. Let $\mathcal{A}$ be a Furstenberg family. 

If for some (equivalently, for all) sequences $(\varepsilon_p)_{p\geq 1}$ of positive numbers tending to $0$ there is a sequence $(A_p)_{p\geq 1}$ of pairwise disjoint sets in $\mathcal{A}$ such that, for any $p,q\geq 1$, any $m\in A_q$, and $j=-p,\ldots,p$,
\begin{equation}\label{eq1}
\Vertiii{\sum_{\substack{n\in A_p\\n\neq m}} e_{n-m+j}} <\min(\varepsilon_p,\varepsilon_q),
\end{equation}
then $B:\widehat{X}\to \widehat{X}$ is $\mathcal{A}$-hypercyclic for $X$. If $(e_n)_{n\in \IZ}$ is an unconditional basis in $X$ then the converse also holds.

If $B$ is invertible, then it suffices to require \eqref{eq1} for $j=0$.
\end{theorem}

The proof is similar to that of \cite[Theorem 6.1]{BoGr17}.

\begin{proof} In order to see that the assumption `for some sequences $(\varepsilon_p)_{p\geq 1}$' implies the assumption `for all sequences $(\varepsilon_p)_{p\geq 1}$' it suffices to pass to a subsequence of $(A_p)_{p\geq 1}$. 

For the final assertion of the theorem, let $M=\max(\vertiii{B}, \vertiii{B^{-1}})$. We then choose $(\varepsilon_p)_{p\geq 1}$ in such a way that $(\widetilde{\varepsilon}_p):=(M^p\varepsilon_p)$ converges monotonically to zero. Hence, for $j=-p,\ldots,p$,
\[
\Vertiii{\sum _{\substack{n\in A_p\\n\neq m}} e_{n-m+j}} \leq M^p \Vertiii{\sum _{\substack{n\in A_p\\n\neq m}} e_{n-m}}<M^p\min(\varepsilon_p,\varepsilon_q) \leq \min(\widetilde{\varepsilon}_p,\widetilde{\varepsilon}_q),
\]
and the claim follows.

We first show that, in addition to the hypothesis, we may assume that, for an arbitrarily fixed sequence $(\alpha_p)_{p\geq 1}$ of positive numbers we have, for any $p\geq 1$, 
\begin{equation}\label{eq2}
\Vertiii{\sum_{\substack{n\in A_p}} e_{n+p}} <\alpha_p.
\end{equation}
Indeed, we start with sequences $(\varepsilon_p)$ and $(A_p)$. We fix an element $m\in A_1$, $m>0$. Then there exists a strictly increasing sequence $(r_p)_{p\geq 1}$ of positive integers such that $r_1>1$, $r_p\geq m+p$ and $\varepsilon_{r_p}\leq \min(\varepsilon_p,\alpha_p)$ for $p\geq 1$.  Let 
\[
\widetilde{A}_p = A_{r_p},  \quad p\geq 1.
\]
Then in view of $m+p\leq r_p$ we have that
\[
\Vertiii{\sum _{\substack{n\in \widetilde{A}_p}} e_{n+p}} = \Vertiii{\sum _{\substack{n\in A_{r_p}}} e_{n-m+(m+p)}}<\varepsilon_{r_p}\leq \alpha_p.
\]
And condition \eqref{eq1} remains valid for the sequence $(\widetilde{A}_p)$. Thus we may assume \eqref{eq2} in addition to \eqref{eq1}. 

We will now show that $B:\widehat{X}\to \widehat{X}$ is $\mathcal{A}$-hypercyclic for $X$. We set, for $p\geq 1$,
\[
\alpha_p = \frac{1}{2^pp \sum_{j=-p}^p \vertiii{B}^{p-j}}\quad\text{and}\quad \varepsilon_p = \frac{1}{p(2p+1)4^p}.
\]
Let $(A_p)_{p\geq 1}$ be a sequence of pairwise disjoint sets in $\mathcal{A}$ such that \eqref{eq1} and \eqref{eq2} hold.

Since $(e_n)_{n\in\IZ}$ is a basis in $X$, there is a dense sequence $(y^{(p)})_{p\geq 1}$ of elements in $X$ of the form
\[
y^{(p)} = \sum_{j=-p}^p y^{(p)}_j e_j,\quad \max_{-p\leq j\leq p}|y_j^{(p)}|\leq p.
\]
We set 
\[
x=\sum_{p=1}^\infty  \sum_{j=-p}^p y^{(p)}_j \sum_{n\in A_p} e_{n+j}
\]
(only the interior sum is formal) and claim that $x$ is a well-defined element of $\widehat{X}$ that is $\mathcal{A}$-hypercyclic for $X$. Indeed, by \eqref{eq2}, 
\[
\sum_{j=-p}^p y^{(p)}_j B^{p-j}\sum_{n\in A_p} e_{n+p}= \sum_{j=-p}^p y^{(p)}_j \sum_{n\in A_p}e_{n+j}
\]
belongs to $\widehat{X}$. Moreover,
\begin{align*}
\Vertiii{ \sum_{j=-p}^p y^{(p)}_j \sum_{n\in A_p}e_{n+j}}&\leq 
\sum_{j=-p}^p |y^{(p)}_j| \Vertiii{B^{p-j}\sum_{n\in A_p} e_{n+p}} \\&\leq 
p \sum_{j=-p}^p \vertiii{B}^{p-j}\Vertiii{\sum_{n\in A_p} e_{n+p}}\leq p \sum_{j=-p}^p \vertiii{B}^{p-j}\alpha_p= \frac{1}{2^p},
\end{align*}
which implies that $x$ defines an element in $\widehat{X}$.

Now let $q\geq 1$. Then we have for any $m\in A_q$ that
\[
B^mx-y^{(q)} = \sum_{p=1}^\infty  \sum_{j=-p}^p y^{(p)}_j \sum_{n\in A_p}e_{n-m+j}-\sum_{j=-q}^q y^{(q)}_je_j= \sum_{p=1}^\infty  \sum_{j=-p}^p y^{(p)}_j \sum_{\substack{n\in A_p\\n\neq m}}e_{n-m+j}; 
\]
note that the terms $n= m$ disappear: if $p\neq q$ then use the fact that $A_p$ and $A_q$ are disjoint; if $p=q$ then the term with $n=m$ cancels. By \eqref{eq1} and the choice of $(\varepsilon_p)$ we have that for $j=-p,\ldots,p$, if $p<q$ then
\[
\Vertiii{\sum_{\substack{n\in A_p\\n\neq m}}  e_{n-m+j}} < \frac{1}{q(2q+1)4^q}\leq \frac{1}{p(2p+1)2^p2^q}, 
\]
while if $p\geq q$ then 
\[
\Vertiii{\sum_{\substack{n\in A_p\\n\neq m}}  e_{n-m+j}} < \frac{1}{p(2p+1)4^p}\leq\frac{1}{p(2p+1)2^p2^q}.
\]
Altogether we have that, for any $q\geq 1$ and $m\in A_q$, 
\[
\vertiii{B^mx-y^{(q)}} \leq  \sum_{p=1}^\infty \sum_{j=-p}^p |y^{(p)}_j|\Vertiii{\sum_{\substack{n\in A_p\\n\neq m}}  e_{n-m+j}}\leq \sum_{p=1}^\infty\frac{1}{2^p2^q} = \frac{1}{2^q},
\]
so that $x$ is $\mathcal{A}$-hypercyclic for $X$.

For the converse we now suppose that $(e_n)_{n\in\IZ}$ is an unconditional basis of $X$. We will need the following properties. 
\begin{itemize}
	\item[(F)] There is a sequence $(\beta_p)_{p\geq 1}$ of positive numbers such that, for any $p\geq 1$, $\vertiii{x}<\beta_p$ implies that $|x_{j}|< 1$ for $j=-p,\ldots,p$.
\end{itemize}
This follows from the continuity of the coordinate projections.
\begin{itemize}
	\item[(M)] If $\sum_{n\in\IZ}x_ne_n\in \widehat{X}$ and $(a_n)_{n\in\IZ}$ is a bounded sequence of scalars then $\sum_{n\in\IZ}a_nx_ne_n\in \widehat{X}$. Moreover, there is an absolute constant $C>0$ such that 
\[	
\Vertiii{\sum_{n\in\IZ}a_nx_ne_n}\leq C \sup_{n\in\IZ}|a_n|\, \Vertiii{\sum_{n\in\IZ}x_ne_n}.
\]
\end{itemize}
This is, for $x\in X$, a consequence of the unconditionality of the basis and extends directly to all of $\widehat{X}$.

Let $x\in \widehat{X}$ be an $\mathcal{A}$-hypercyclic vector for $X$, and let $(\varepsilon_p)_{p\geq 1}$ be a decreasing sequence of positive numbers with $\varepsilon_p\leq 1$, $p\geq 1$. Let $(\beta_p)_{p\geq 1}$ and $C>0$ be given by properties (F) and (M). We then define inductively numbers $C_p\geq \max(1,C)$ for $p\geq 1$ such that if $q=p\geq 1$ with $-p\leq j < k\leq p$, or if $q>p\geq 1$ with $-p\leq j\leq p$ and $-q\leq k\leq q$, then
\begin{equation}\label{eq-dpf}
(k+q+1)\frac{C_q}{\varepsilon_q}-(j+p+1)\frac{C_p}{\varepsilon_p}\geq 2.
\end{equation}

Let
\[
y^{(p)} = \sum_{j=-p}^p\Big(1+(j+p+1)\frac{C_p}{\varepsilon_p}\Big) e_j\in X,\quad p\geq 1,
\]
and choose numbers $\rho_p>0$ such that the open balls of radius $\rho_p$ in $\widehat{X}$ around $y^{(p)}$, $p\geq 1$, are pairwise disjoint. Then, by $\mathcal{A}$-hypercyclicity of $x$, there are sets $A_p\in\mathcal{A}$, $p\geq 1$, such that, for all $n\in A_p$, 
\begin{equation}\label{eq-dpf2}
\vertiii{B^nx-y^{(p)}}<\min \Big(\rho_p,\beta_{p},\frac{\varepsilon_p}{C_p}\Big).
\end{equation}
It follows that the sets $A_p$, $p\geq 1$, are pairwise disjoint. We will show that they satisfy the hypothesis of the theorem.

We deduce from \eqref{eq-dpf2} with (F) that, for any $n\in A_p$ and $j=-p,\ldots,p$,
\begin{equation}\label{eq-dp}
\Big|x_{n+j}-\Big(1+(j+p+1)\frac{C_p}{\varepsilon_p}\Big)\Big|<1,
\end{equation}
hence
\begin{equation}\label{eq-dp2}
\Big|\frac{1}{x_{n+j}}\Big|<\frac{\varepsilon_p}{(j+p+1)C_p}\leq \varepsilon_p. 
\end{equation}

Now let $p,q\geq 1$, $m\in A_q$ and $j=-p,\ldots,p$. Then 
\[
\vertiii{B^mx-y^{(q)}} < \frac{\varepsilon_q}{C_q},
\]
and hence with \eqref{eq-dp2} in view of (M)
\begin{equation}\label{eq-M}
\Vertiii{\sum_{\substack{n\in A_p\\n\neq m}} \frac{1}{x_{n+j}}[B^mx-y^{(q)}]_{n-m+j} e_{n-m+j}}<\varepsilon_p\varepsilon_q\leq\min(\varepsilon_p,\varepsilon_q),
\end{equation}
where $[z]_k$ denotes the $k$th entry of the sequence $z$. Let $n\in A_p$, $n\neq m$, and $k=-q,\ldots,q$. Then by \eqref{eq-dp} we have that
\[
\Big|x_{n+j} -x_{m+k}-(j+p+1)\frac{C_p}{\varepsilon_p}+(k+q+1)\frac{C_q}{\varepsilon_q}\Big|<2.
\]
This implies that $n+j\neq m+k$. Indeed, equality can hold by \eqref{eq-dpf} only if $p=q$ and $j=k$, which is impossible since $m\neq n$. Thus we have that $|n-m+j|>q$ and therefore
\[
[B^mx-y^{(q)}]_{n-m+j}=x_{n+j}.
\]
Hence \eqref{eq-M} reduces to the hypothesis of the theorem.
\end{proof}

\begin{remark}\label{r-spread}
For later use we note that one may assume in addition that, for any $p,q\geq 1$,
\[
\min_{\substack{n\in A_p, m\in A_q\\n\neq m}}|n-m|>p+q.
\]
Indeed, let us choose $\varepsilon_p\leq \beta_p$ for $p\geq 1$, where the $\beta_p$ satisfy the condition in (F). Then \eqref{eq1} implies that, for any $p,q\geq 1$, any $m\in A_q$, $k=-q,\ldots,q$, any $n\in A_p$, $j=-p,\ldots,p$,
\[
n-m+j\neq k,
\]
in other words $|n-m|> p+q$.
\end{remark}

General weighted shifts can now be treated in the way described in \cite[Section 4.1]{GrPe11}. Let $B_w$ be a bilateral weighted backward shift operator on the sequence space $X$. We define $v=(v_n)_{n\in\mathbb{Z}}$ by
\begin{equation}\label{vns}
v_n=\begin{cases} \frac{1}{w_1\cdots w_n},& \text{for $n>0$},\\
1,& \text{for $n=0$},\\
{w_{n+1}\cdots w_0},& \text{for $n < 0$},
\end{cases}
\end{equation}
and we set
\[
X(v)=\{(x_n) : (x_nv_n)\in X\}.
\]
Then  $B_w$ on $X$ is conjugate to $B$ on $X(v)$ via the bijection $\phi_v: X(v)\to X$, $(x_n)\to (x_nv_n)$, see \cite[pp.~100-101]{GrPe11}.

\begin{theorem}\label{t-bwahcgen}
Let $X$ be a Banach sequence space over $\IZ$ in which $(e_n)_{n\in \IZ}$ is a basis. Suppose that the weighted backward shift $B_w$ is an operator on $X$, and define $v$ by \eqref{vns}. Let $\mathcal{A}$ be a Furstenberg family. 

If for some (equivalently, for all) sequences $(\varepsilon_p)_{p\geq 1}$ of positive numbers tending to $0$ there is a sequence $(A_p)_{p\geq 1}$ of pairwise disjoint sets in $\mathcal{A}$ such that, for any $p,q\geq 1$, any $m\in A_q$, and $j=-p,\ldots,p$,
\begin{equation}\label{eq10}
\Vertiii{\sum_{\substack{n\in A_p\\n\neq m}} v_{n-m+j} e_{n-m+j}} <\min(\varepsilon_p,\varepsilon_q),
\end{equation}
then $B_w:\widehat{X}\to \widehat{X}$ is $\mathcal{A}$-hypercyclic for $X$. If $(e_n)_{n\in \IZ}$ is an unconditional basis in $X$ then the converse also holds.

If $B_w$ is invertible, then it suffices to require \eqref{eq10} for $j=0$.
\end{theorem}

\begin{remark}\label{r-equivcond}
The hypothesis can be rephrased equivalently in a more natural way:

\textit{For some (equivalently, for all) sequences $(\varepsilon_p)_{p\geq 1}$ of positive numbers tending to $0$ there is a sequence $(A_p)_{p\geq 1}$ of pairwise disjoint sets in $\mathcal{A}$ such that
\begin{itemize}
\item[\rm (i)] for any $p,q\geq 1$
\[
\min_{\substack{n\in A_p, m\in A_q\\n\neq m}}|n-m|>p+q;
\]
\item[\rm (ii)] for any $p,q\geq 1$, any $m\in A_q$, and any $j=-p,\ldots,p$,
\[
\Vertiii{\sum _{\substack{n\in A_p\\n< m}} \Big(\prod_{\nu=n-m+j+1}^{0}w_\nu\Big) e_{n-m+j}} <\min(\varepsilon_p,\varepsilon_q)
\]
and
\[
\Vertiii{\sum _{\substack{n\in A_p\\n> m}} \frac{1}{\prod_{\nu=1}^{n-m+j}w_\nu}e_{n-m+j}} <\min(\varepsilon_p,\varepsilon_q).
\]
\end{itemize}
}
Indeed, by Remark \ref{r-spread}, one may assume condition (i) without loss of generality. In that case we have, whenever $m\in A_q$, $n\in A_p$ and $j=-p,\ldots,p$, that $n-m+j<0$ if $n< m$ and $n-m+j>0$ if $n> m$. Thus the single norm in \eqref{eq10} can be split equivalently into the two norms above, up to an irrelevant factor 2.
\end{remark}

\section{From hypercyclicity \textit{for} $X$ to hypercyclicity \textit{in} $X$}\label{s-fortoin}

Theorem \ref{t-bwahcgen} provides a characterisation of when a weighted backward shift on $\widehat{X}$ is $\mathcal{A}$-hypercyclic for $X$. 
We will improve this result in the case when $\mathcal{A}$ is an upper Furstenberg family. This will follow from a more general result for arbitrary operators on Banach spaces that we derive in this section; this result might also be of independent interest.

Let us first recall the notion of upper Furstenberg family, which was recently introduced in \cite{BoGr17}.

\begin{definition}\label{d-upperfur}
A Furstenberg family $\mathcal{A}$ is called \emph{upper} if it does not contain the empty set and it can be written as
\[
\mathcal{A} = \bigcup_{\delta\in D} \mathcal{A}_{\delta}\quad \mbox{with}\quad \mathcal{A}_{\delta}:=\bigcap_{\mu\in M} \mathcal{A}_{\delta,\mu}
\]
for some families $\mathcal{A}_{\delta,\mu}$ ($\delta\in D,\mu\in M$), where $D$ is arbitrary but $M$ is countable, and such that
\begin{itemize}
\item[(i)] for any $A\in \mathcal{A}_{\delta,\mu}$ there is a finite set $F\subset\mathbb{N}_0$ such that if $B\supset A\cap F$ then $B\in\mathcal{A}_{\delta,\mu}$;
\item[(ii)] for any $A\in \mathcal{A}$ there is some $\delta\in D$ such that, for all $n\geq 0$, $A-n\in \mathcal{A}_\delta$.
\end{itemize}
\end{definition}

For example, the family of sets of positive upper density is an upper Furstenberg family. Thus the results in this section hold, in particular, for upper frequent hypercyclicity.

In \cite{BoGr17} the authors have obtained a Birkhoff type theorem for upper Furstenberg families; the proof is by a simple application of the Baire category theorem. We will need here only the following implication. Let $T$ be an operator on a separable Banach space $X$ and $\mathcal{A} = \bigcup_{\delta\in D}\bigcap_{\mu\in M} \mathcal{A}_{\delta,\mu}$ an upper Furstenberg family. If
\begin{itemize}
	\item[(B)] for any non-empty open subset $V$ of $X$ there is some $\delta\in D$ such that for any non-empty open subset $U$ of $X$ and any $\mu\in M$ there is some $x\in U$ such that
\[
\{n\geq 0 : T^n x\in V\} \in \mathcal{A}_{\delta,\mu},
\]
\end{itemize}
then $T$ is $\mathcal{A}$-hypercyclic.  

In addition, we need a necessary condition for $\mathcal{A}$-hypercyclicity for a subspace.

\begin{proposition}\label{p-Birksub}
Let $Y$ be a Banach space, $X$ a closed subspace of $Y$, and $T$ an operator on $Y$. Let $\mathcal{A} = \bigcup_{\delta\in D}\bigcap_{\mu\in M} \mathcal{A}_{\delta,\mu}$ be an upper Furstenberg family. If $T$ is $\mathcal{A}$-hypercyclic for $X$ then, for any open subset $V$ of $Y$ with $V\cap X\neq \varnothing$, there is some $\delta\in D$ such that for any open subset $U$ of $Y$ with $U\cap X\neq \varnothing$ there is some $y\in U$ such that
\[
\{n\geq 0 : T^n y\in V\} \in \mathcal{A}_{\delta}.
\]
\end{proposition}

\begin{proof} We fix an $\mathcal{A}$-hypercyclic vector $y'$ for $X$ under $T$. Let $V$ be an open subset of $Y$ with $V\cap X\neq \varnothing$. Then
\[
\{n\geq 0 : T^n y' \in V\} \in \mathcal{A}.
\]
By property (ii) of upper Furstenberg families there is some $\delta\in D$ such that, for all $m\geq 0$,
\[
\{n\geq 0 : T^n y' \in V\}-m \in \mathcal{A}_{\delta}.
\]
Now let $U$ be an open subset of $Y$ with $U\cap X\neq \varnothing$. Again by $\mathcal{A}$-hypercyclicity there is some $m\geq 0$ such that $y:=T^my'\in U$, and therefore
\[
\{n\geq 0 : T^ny\in V\} = \{k\geq 0 : T^ky'\in V\}-m \in \mathcal{A}_{\delta},
\]
which had to be shown.
\end{proof}

We can now derive the announced result.

\begin{theorem}\label{t-transfer}
Let $T$ be an operator on a Banach space $Y$ and $X$ a closed subspace of $Y$ that is invariant under $T$. Let $\mathcal{A}$ be an upper Furstenberg family. Suppose that the following property holds: 

For all $x_2\in X$ and $\varepsilon_2>0$ there are $x_2'\in X$ and $\eta_2>0$ such that for all $x_1\in X$ and $\varepsilon_1>0$ there are $x'_1\in X$ and $\eta_1>0$ so that for all $y\in Y$ and all finite subsets $F\subset \mathbb{N}_0$, if
\begin{equation}\label{eq-transfer1}
\|x'_1-y\|<\eta_1 \text{ and }  \|x'_2-T^ny\|<\eta_2 \text{ for all $n\in F$} 
\end{equation}
then there is some $x\in X$ such that
\begin{equation}\label{eq-transfer2}
\|x_1-x\|<\varepsilon_1 \text{ and }  \|x_2-T^nx\|<\varepsilon_2 \text{ for all $n\in F$}.
\end{equation}
Then the following assertions are equivalent:
\begin{itemize}
\item[\emph{(a)}] $T:Y\to Y$ is $\mathcal{A}$-hypercyclic for $X$. 
\item[\emph{(b)}] $T:X\to X$ is $\mathcal{A}$-hypercyclic. 
\end{itemize}
\end{theorem}

\begin{proof} Since (b) trivially implies (a) we need only deduce (b) from (a).

We write $\mathcal{A} = \bigcup_{\delta\in D}\bigcap_{\mu\in M} \mathcal{A}_{\delta,\mu}$ according to Definition \ref{d-upperfur}.
Let $T:Y\to Y$ be $\mathcal{A}$-hypercyclic for $X$. We want to show that $T:X\to X$ satisfies the Birkhoff type condition (B) stated above.
To see this, let $V$ be a non-empty open subset of $X$. Without loss of generality we may assume that $V=B(x_2,\varepsilon_2)$, the open ball of radius $\varepsilon_2>0$ in $X$ around a point $x_2\in X$. Let $V_Y=B(x'_2,\eta_2)$ be the open ball taken in $Y$, where $x'_2\in X$ and $\eta_2>0$ are given by the hypothesis. Let $\delta>0$ be a value associated to $V_Y$ according to Proposition \ref{p-Birksub}. 

Next, let $U$ be a non-empty open subset of $X$. Without loss of generality we may assume that $U=B(x_1,\varepsilon_1)$ with $x_1\in X$ and $\varepsilon_2>0$. Let $U_Y=B(x'_1,\eta_1)$ be the ball taken in $Y$, where $x'_1\in X$ and $\eta_1>0$ are given by the hypothesis.

Now, by Proposition \ref{p-Birksub}, for any $\mu\in M$ there is some $y\in U_Y$ such that
\[
\{n\geq 0 : T^n y\in V_Y\} \in \mathcal{A}_{\delta,\mu}
\]
(in fact, $y$ may even be chosen independently of $\mu$, but that will not be used). Note that the point $y$ lies in $Y$; it remains to show that it can be replaced by a point $x$ in $X$ that satisfies (B). 

Thus, let $\mu\in M$. By property (i) of upper Furstenberg families there is a finite set $G\subset\mathbb{N}_0$ such that any superset of 
\[
F:=\{n\geq 0 : T^n y\in V_Y\}\cap G
\]
belongs to $\mathcal{A}_{\delta,\mu}$. Thus condition \eqref{eq-transfer1} holds. By hypothesis there is some $x\in X$ such that \eqref{eq-transfer2} holds. We then have that $x\in B(x_1,\varepsilon_1)=U$, and the set 
\[
\{n\geq 0 : T^n x\in V\} = \{n\geq 0 : \|x_2-T^nx\|<\varepsilon_2\}
\]
contains $F$ and thus belongs to $\mathcal{A}_{\delta,\mu}$.

Altogether we have shown that condition (B) holds, which implies (b).
\end{proof}

One may understand condition (b) as saying that the operator $T: Y\to Y$ admits an $\mathcal{A}$-hypercyclic vector for $X$ that comes from $X$. Thus our result is similar in spirit to Theorem 2.1 of Herzog \cite{Hzg94}.

\section{Bilateral shifts on $X$ for upper Furstenberg families}\label{s-bilwsupp}

The previous general theorem easily implies an improvement of Theorem \ref{t-bwahcgen} for upper Furstenberg families (under the assumption of unconditionality of the basis).

\begin{theorem}\label{t-bwahcupper}
Let $X$ be a Banach sequence space over $\IZ$ in which $(e_n)_{n\in \IZ}$ is an unconditional basis. Suppose that the weighted backward shift $B_w$ is an operator on $X$, and define $v$ by \eqref{vns}. Let $\mathcal{A}$ be an upper Furstenberg family. Then the following assertions are equivalent.

\begin{itemize}
\item[\emph{(a)}] For some (equivalently, for all) sequences $(\varepsilon_p)_{p\geq 1}$ of positive numbers tending to $0$ there is a sequence $(A_p)_{p\geq 1}$ of pairwise disjoint sets in $\mathcal{A}$ such that, for any $p,q\geq 1$, any $m\in A_q$, and $j=-p,\ldots,p$,
\begin{equation}\label{eq6}
\Vertiii{\sum_{\substack{n\in A_p\\n\neq m}} v_{n-m+j}e_{n-m+j}} <\min(\varepsilon_p,\varepsilon_q).
\end{equation}
\item[\emph{(b)}] $B_w:\widehat{X}\to \widehat{X}$ is $\mathcal{A}$-hypercyclic for $X$. 
\item[\emph{(c)}] $B_w:X\to X$ is $\mathcal{A}$-hypercyclic. 
\end{itemize}

\noindent If $B_w$ is invertible, then it suffices to require \eqref{eq6} for $j=0$.
\end{theorem}

\begin{remark}\label{r-equivcondrem} 
Remark \ref{r-equivcond} retains its validity. 
\end{remark}

\begin{proof} In view of Theorems \ref{t-bwahcgen} and \ref{t-transfer} we need only show that $B_w$ satisfies the hypothesis of Theorem \ref{t-transfer} with $Y=\widehat{X}$. Thus, let $x^{(1)}, x^{(2)}\in X$ and $\varepsilon_1,\varepsilon_2>0$ be given (note that we prove a stronger version of the hypothesis). By density we may assume that $x^{(1)}$ and $x^{(2)}$ are finitely non-zero sequences.

We set $\eta_1= \varepsilon_1/C$, $\eta_2= \varepsilon_2/C$, where $C$ is a constant for which condition (M) above is satisfied. Let $\widehat{x}\in \widehat{X}$ and $F\subset \mathbb{N}_0$ be finite such that
\[
\vertiii{x^{(1)}-\widehat{x}}<\eta_1=\varepsilon_1/C \text{ and }  \vertiii{x^{(2)}-B_w^n\widehat{x}}<\eta_2=\varepsilon_2/C \text{ for all $n\in F$}.
\] 
Suppose that
\[
\widehat{x}= \sum_{n\in\IZ} x_ne_n.
\]
Then, for $l\geq 0$, we define
\[
x = \sum_{|n|\leq l} x_ne_n,
\]
which belongs to $X$ as a finitely non-zero sequence. Since $x^{(1)}$ is finitely non-zero, if $l$ is sufficiently large, then $x^{(1)}-x$ can be obtained from $x^{(1)}-\widehat{x}$ by deleting entries. Similarly, for any $n\in \mathbb{N}_0$, if $l$ is sufficiently large, then $x^{(2)}-B_w^n x$ can be obtained from $x^{(2)}-B_w^n \widehat{x}$ by deleting entries. It then follows from (M) that for all sufficiently large $l$, 
\[
\vertiii{x^{(1)}-x}<\varepsilon_1 \text{ and }  \vertiii{x^{(2)}-B_w^nx}<\varepsilon_2 \text{ for all $n\in F$};
\] 
note that we have used here that $F$ is a finite set. Thus $B_w$ satisfies the hypothesis of Theorem \ref{t-transfer}.
\end{proof}

Keeping Remark \ref{r-equivcondrem} in mind we spell out the special case of the theorem for $X=c_0(\IZ)$. 

\begin{corollary}\label{c-bwahc0}
Let $B_w$ be a weighted backward shift on $c_0(\IZ)$, and let $\mathcal{A}$ be an upper Furstenberg family. Then the following assertions are equivalent.

\begin{itemize}
\item[\emph{(a)}] For some (equivalently, for all) sequences $(\varepsilon_p)_{p\geq 1}$ of positive numbers tending to $0$ there is a sequence $(A_p)_{p\geq 1}$ of pairwise disjoint sets in $\mathcal{A}$ such that
\begin{itemize}
\item[\rm (i)] for any $p,q\geq 1$
\[
\min_{\substack{n\in A_p, m\in A_q\\n\neq m}}|n-m|>p+q;
\]
\item[\rm (ii)] for any $p,q\geq 1$, any $m\in A_q$, $n\in A_p$, any $j=-p,\ldots,p$, if $n<m$ then
\[
\prod_{\nu=n-m+j+1}^{0}|w_\nu| <\min(\varepsilon_p,\varepsilon_q)
\]
and if $n>m$ then
\[
\prod_{\nu=1}^{n-m+j}|w_\nu|>\frac{1}{\min(\varepsilon_p,\varepsilon_q)}.
\]
\end{itemize}
\item[\emph{(b)}] $B_w:\ell^\infty(\IZ)\to \ell^\infty(\IZ)$ is $\mathcal{A}$-hypercyclic for $c_0(\IZ)$. 
\item[\emph{(c)}] $B_w:c_0(\IZ)\to c_0(\IZ)$ is $\mathcal{A}$-hypercyclic. 
\end{itemize}

\noindent If $B_w$ is invertible, then it suffices to require $j=0$ in condition \emph{(ii)}.
\end{corollary}

This result should be compared with the case of upper frequent hypercyclicity in \cite[Theorem 12]{BaRu15}. First of all, the result of Bayart and Ruzsa is now extended to not necessarily invertible operators and to arbitrary upper Furstenberg families. But the main improvement is to eliminate the (non-symmetric) condition (c) in their result. In fact, removing this condition has been the main goal of our work here: this leads to the following result, which contains Theorem \ref{t-main} as a special case.

\begin{theorem}\label{t-main2}
Let $B_w$ be an invertible weighted backward shift on $c_0(\IZ)$ and $\mathcal{A}$ be an upper Furstenberg family. If $B_w$ is $\mathcal{A}$-hypercyclic then so is $B_w^{-1}$.
\end{theorem}

\begin{proof}
The inverse of $B_w$ is the weighted forward shift $F_{1/w}: y\to (\frac{1}{w_{n}}y_{n-1})_{n\in\IZ}$. Via the bijection $\phi(x)=(x_{-n})$ on $c_0(\IZ)$, this forward shift is conjugate to the weighted backward shift $B_{w'}$ with
\[
w'_n=\frac{1}{w_{-n+1}},\quad n\in\IZ,
\]
that is, we have that $\phi\circ B_{w'}=F_{1/w}\circ \phi$. Since the characterising condition in Corollary \ref{c-bwahc0} is invariant under the passage from $w$ to $w'$ (it suffices to consider $j=0$) the result follows.
\end{proof}

\section{Bilateral shifts on $X$ for arbitrary Furstenberg families}\label{s-bilwsXarb}

As a by-product of Theorem \ref{t-bwahcgen} we may also obtain a characterisation of the $\mathcal{A}$-hypercyclicity of weighted backward shifts on $X$ itself for arbitrary Furstenberg families, thereby providing an analogue of the result in \cite{BoGr17} for unilateral shifts. The price we have to pay is the introduction of an additional condition.

\begin{theorem}\label{t-bwahcgenX}
Let $X$ be a Banach sequence space over $\IZ$ in which $(e_n)_{n\in \IZ}$ is a basis. Suppose that the weighted backward shift $B_w$ is an operator on $X$, and define $v$ by \eqref{vns}. Let $\mathcal{A}$ be a Furstenberg family. 

If for some (equivalently, for all) sequences $(\varepsilon_p)_{p\geq 1}$ of positive numbers tending to $0$ there is a sequence $(A_p)_{p\geq 1}$ of pairwise disjoint sets in $\mathcal{A}$ such that
\begin{itemize}
\item[\emph{(i)}] for any $p\geq 1$
\[
\sum_{n\in A_p} v_{n+p}e_{n+p} \quad \text{converges in $X$};
\]
\item[\emph{(ii)}] for any $p,q\geq 1$, any $m\in A_q$, and $j=-p,\ldots,p$,
\begin{equation}\label{eq11}
\Big\|\sum_{\substack{n\in A_p\\n\neq m}} v_{n-m+j} e_{n-m+j}\Big\| <\min(\varepsilon_p,\varepsilon_q),
\end{equation}
\end{itemize}
then $B$ is $\mathcal{A}$-hypercyclic. If $(e_n)_{n\in \IZ}$ is an unconditional basis in $X$ then the converse also holds.

If $B_w$ is invertible, then it suffices to require \eqref{eq11} for $j=0$; and in \emph{(i)} one may demand that $\sum_{n\in A_p} v_{n}e_{n}$ converges in $X$ for any $p\geq 1$.
\end{theorem}

\begin{remark}\label{r-equivcondrem2} 
Once more we recall Remark \ref{r-equivcond} for an equivalent statement of condition (ii). 
\end{remark}

\begin{proof} It suffices to consider the unweighted case when $w_n=v_n=1$ for all $n\in \IZ$. 

For sufficiency we proceed exactly as in the proof of Theorem \ref{t-bahcgen}; the sole difference is that condition (i) ensures that 
$\sum_{n\in A_p} e_{n+p} \in X$, so that the whole proof can now be performed in $X$, leading to an $\mathcal{A}$-hypercyclic vector $x\in X$.
Incidentally, the fact that the series in (ii) converge follows from (i) and an application of a suitable power of $B_w$.

Likewise for necessity we can repeat the proof of Theorem \ref{t-bahcgen} verbatim. Since, now, $x\in X$ we have that, for any $p\geq 1$,
\[
\sum_{n\in\IZ} x_{n+p}e_{n+p}
\]
converges in $X$. Since, by \eqref{eq-dp2},
\[
\Big|\frac{1}{x_{n+p}}\Big|\leq \varepsilon_p
\]
for any $n\in A_p$, unconditionality of the basis then implies (i) in addition to (ii). 
\end{proof} 

In view of Question \ref{q1} it may be of interest to state explicitly the result for $X=c_0(\IZ)$, where we take account of Remark \ref{r-equivcondrem2}.

\begin{corollary}\label{c-bwahcgenX}
Let $B_w$ be a weighted backward shift on $c_0(\IZ)$, and let $\mathcal{A}$ be a  Furstenberg family. Then the following assertions are equivalent.
\begin{itemize}
\item[\emph{(a)}] For some (equivalently, for all) sequences $(\varepsilon_p)_{p\geq 1}$ of positive numbers tending to $0$ there is a sequence $(A_p)_{p\geq 1}$ of pairwise disjoint sets in $\mathcal{A}$ such that
\begin{itemize}
\item[\rm (i)] for any $p,q\geq 1$
\[
\min_{\substack{n\in A_p, m\in A_q\\n\neq m}}|n-m|>p+q;
\]
\item[\rm (ii)] for any $p\geq 1$,
\[
w_1\cdots w_{n+p}\to \infty\quad\text{as $n\to\infty$, $n\in A_p$};
\]
\item[\rm (iii)] for any $p,q\geq 1$, any $m\in A_q$, $n\in A_p$, any $j=-p,\ldots,p$, if $n<m$ then
\[
\prod_{\nu=n-m+j+1}^{0}|w_\nu| <\min(\varepsilon_p,\varepsilon_q)
\]
and if $n>m$ then
\[
\prod_{\nu=1}^{n-m+j}|w_\nu|>\frac{1}{\min(\varepsilon_p,\varepsilon_q)}.
\]
\end{itemize}
\item[\rm (b)] $B_w$ is $\mathcal{A}$-hypercyclic on $c_0(\IZ)$.
\end{itemize}
If $B_w$ is invertible, then it suffices to require $j=0$ in condition \emph{(iii)}; and in \emph{(ii)} one may demand that, for any $p\geq 1$, $w_1\cdots w_{n}\to \infty$ as $n\to\infty$, $n\in A_p$.
\end{corollary}

This extends the case of frequent hypercyclicity in \cite[Theorem 12]{BaRu15} to not necessarily invertible operators and to arbitrary Furstenberg families. 

But the question posed by Bayart and Ruzsa \cite[p.~707]{BaRu15} remains: can one drop the (non-symmetric) condition (ii) for frequent hypercyclicity (as we did in the previous section for upper Furstenberg families)? If so then the bilateral weighted shifts on $c_0(\IZ)$ could not serve as counter-examples to Question \ref{q1}.

\section{Unilateral shifts}\label{s-uni}
We may apply the same techniques as in Sections \ref{s-bilwsarb} and \ref{s-bilwsupp} to unilateral weighted backward shifts on sequence spaces over $\IN_0$, which leads to the following.

\begin{theorem}\label{t-bwahcupperuni}
Let $X$ be a Banach sequence space over $\IN_0$ in which $(e_n)_{n\geq 0}$ is a basis. Suppose that the unilateral weighted backward shift $B_w$ is an operator on $X$. Let $\mathcal{A}$ be a Furstenberg family. Consider the following assertions:

\begin{itemize}
\item[\emph{(a)}] For some (equivalently, for all) sequences $(\varepsilon_p)_{p\geq 1}$ of positive numbers tending to $0$ there is a sequence $(A_p)_{p\geq 1}$ of pairwise disjoint sets in $\mathcal{A}$ such that, for any $p,q\geq 1$, any $m\in A_q$, and $j=0,\ldots,p$,
\[
\Vertiii{\sum_{\substack{n\in A_p\\n> m}} \frac{1}{\prod_{\nu=1}^{n-m+j}w_\nu}e_{n-m+j}} <\min(\varepsilon_p,\varepsilon_q).
\]
\item[\emph{(b)}] $B_w:\widehat{X}\to \widehat{X}$ is $\mathcal{A}$-hypercyclic for $X$. 
\item[\emph{(c)}] $B_w:X\to X$ is $\mathcal{A}$-hypercyclic. 
\end{itemize}

Then \emph{(a)} implies \emph{(b)}. If $(e_n)_{n\geq 0}$ is an unconditional basis, then \emph{(a)} and \emph{(b)} are equivalent.
If $(e_n)_{n\geq 0}$ is an unconditional basis and $\mathcal{A}$ is an upper Furstenberg family then all three assertions are equivalent.
\end{theorem}

In other words, the Baire category theorem again allows one to drop the convergence condition in \cite[Theorem 6.2]{BoGr17} for upper Furstenberg families. On the other hand, for such families we have seen in \cite[Theorem 5.3]{BoGr17} that, also by Baire, it suffices to have a single set $A$ instead of a sequence $(A_p)_{p\geq 1}$ in \cite[Theorem 6.2]{BoGr17}. Thus one might be tempted to expect that both simplifications are possible at the same time. This, however, is not the case.

\begin{example}
There is a unilateral weighted backward shift $B_w$ on $c_0=c_0(\IN)$ such that, for any $p \geq 0$ and $M > 0$, there exists a subset $A$ of $\IN_0$ of positive upper density such that, for any $n,m \in A$, $n > m$,
\[
|w_1w_2\cdots w_{n-m+p}|>M,
\]
but that is not upper frequently hypercyclic. In other words, $B_w$ satisfies condition (ii) of \cite[Corollary 5.4]{BoGr17}, but not the convergence condition (i) there.

It suffices to consider the example of \cite[Theorem 7.1]{BoGr17}: the weighted backward shift $B_w$ on $c_0$ constructed there is not upper frequently hypercyclic. However, for any $j\geq 0$, the set
\[
A=\{ l.10^j:l\geq 1\}
\] 
has positive (upper) density, and for any $p\geq 0$ and $n,m\in A$ with $n>m$ one has by construction that, for $j\geq p$,
\[
w_1w_2\cdots w_{n-m+p} =\varpi_{n-m+p}\geq 2^{j+p},
\]
which can be made arbitrarily large.
\end{example}

It seems that one may benefit from Baire once, but not twice.

~\\[2mm]
\noindent
\parbox[t][3cm][l]{8cm}{\small
Karl-G. Grosse-Erdmann\\
Département de Mathématique, Institut Complexys\\
Universit\'e de Mons\\
20 Place du Parc\\
7000 Mons, Belgium\\
E-mail: kg.grosse-erdmann@umons.ac.be}

\end{document}